\documentclass[leqno,11pt,a4paper]{article}
\usepackage[left=2.0cm, top=2.2cm,bottom=2.5cm,right=2.0cm]{geometry}
\usepackage{amsmath,amssymb,amsthm,mathrsfs,calc,graphicx}
\usepackage[british]{babel}

\numberwithin{equation}{section}

\newtheorem {theorem}{Theorem}[section]

\newtheorem {lemma}[theorem]{Lemma}
\newtheorem {corollary}[theorem]{Corollary}

{\theoremstyle{definition}

}
{\theoremstyle{theorem}
\newtheorem {remark}[theorem]{Remark}
\newtheorem {example}[theorem]{Example}
}

\def\ba{\begin{array}}
\def\ea{\end{array}}
\def\bea{\begin{eqnarray} \label}
\def\eea{\end{eqnarray}}
\def\be{\begin{equation} \label}
\def\ee{\end{equation}}
\def\bit{\begin{itemize}}
\def\eit{\end{itemize}}
\def\ben{\begin{enumerate}}
\def\een{\end{enumerate}}

\def\EE{\mathbb{E}}

\def\NN{\mathbb{N}}
\def\PP{\mathbb{P}}
\def\QQ{\mathbb{Q}}
\def\RR{\mathbb{R}}

\def\g{\gamma}
\def\d{\delta}

\def\s{\sigma}

\def\O{\Omega}

\def\cF{\mathcal{F}}

\def\cH{\mathcal{H}}

\def\cK{\mathcal{K}}

\def\cZ{\mathcal{Z}}

\def\sZ{\mathscr{Z}}

\def\dint{\textup{d}}

\def\var{{\textup{var}}}
\def\dom{{\textup{dom}}}
\def\s{\star}
\def\ts{\,\widetilde{\star}}

\allowdisplaybreaks

\begin{document}

\title{\bfseries New Berry-Esseen bounds for non-linear functionals of Poisson random measures}

\author{Peter Eichelsbacher\footnotemark[1] \, and  Christoph Th\"ale\footnotemark[2]}

\date{}
\renewcommand{\thefootnote}{\fnsymbol{footnote}}
\footnotetext[1]{Ruhr University Bochum, Faculty of Mathematics, NA 3/66, D-44780 Bochum, Germany. E-mail: peter.eichelsbacher@rub.de}

\footnotetext[2]{Ruhr University Bochum, Faculty of Mathematics, NA 3/68, D-44780 Bochum, Germany. E-mail: christoph.thaele@rub.de}

\maketitle

\begin{abstract}
This paper deals with the quantitative normal approximation of non-linear functionals of Poisson random measures, where the quality is measured by the Kolmogorov distance. Combining Stein's method with the Malliavin calculus of variations on the Poisson space, we derive a bound, which is strictly smaller than what is available in the literature. This is applied to sequences of multiple integrals and sequences of Poisson functionals having a finite chaotic expansion. This leads to new Berry-Esseen bounds in de Jong's theorem for degenerate U-statistics. Moreover, geometric functionals of intersection processes of Poisson $k$-flats, random graph statistics of the Boolean model and non-linear functionals of Ornstein-Uhlenbeck-L\'evy processes are considered.
\end{abstract}

\section{Introduction}

Combining Stein's method with the Malliavin calculus of variations in order to deduce quantitative limit theorems for non-linear functionals of random measures has become a relevant direction of research in recent times. Results in this area usually deal either with functionals of a Gaussian random measure or with functionals of a Poisson random measure. Applications of findings dealing with the Gaussian case have found notable applications in the theory and statistics of Gaussian random processes \cite{BarndorffNielsenEtAl1,BarndorffNielsenEtAl2} (most prominently the fractional Brownian motion \cite{NourdinBool}), spherical random fields \cite{BaldiEtAl,MarinoucciPeccati}, random matrix theory \cite{NourdinPeccatiMatrix} and universality questions \cite{NourdinPeccatiReinert}, whereas the findings for Poisson random measures have attracted applications in stochastic geometry \cite{LRP,LPST,STScaling,STFlats}, non-parametric Bayesian survival analysis \cite{deBlasiEtAl,PeccatiPruenster} or the theory 
of U-statistics \cite{PecTh13,ReitznerSchulte,SchulteKolmogorov}.

The present paper deals with quantitative central limit theorems for Poisson functionals (these are functionals of a Poisson random measure). Whereas most of the existing literature, such as \cite{PSTU2010,PeccatiTaqquDoubleIntegrals,PecTh13,ReitznerSchulte}, deals with smooth distances, such as the Wasserstein distance or a distance based on twice or trice differentiable test functions to measure the quality of the probabilistic approximation, our results deal with the non-smooth Kolmogorov distance. This is the maximal deviation of the involved distribution functions, which we consider as more intuitive and informative; let us agree to call a quantitative central limit theorem using the Kolmogorov distance a Berry-Esseen bound or theorem in what follows. Similar results for the Kolmogorov distance have previously appeared in \cite{SchulteKolmogorov}. Whereas in that paper a bound is derived using a case-by-case study around the non-differentiability point of the solution of the Stein equation and 
an analysis of the second-order derivative, we use a version of Stein's method, which circumvents such a case differentiation completely and avoids the usage of second-order derivatives. This provides new bounds, which differ in parts from those in \cite{SchulteKolmogorov}. In particular, our bounds are strictly smaller and also improve the constants appearing in \cite{SchulteKolmogorov}.

Our general result, Theorem \ref{thm:MalliavinSteinBound} below, is applied to sequences of (compensated) multiple integrals, which are the basic building blocks of the so-called Wiener-It\^o chaos associated with a Poisson random measure. We provide new Berry-Esseen bounds for the normal approximation of such sequences. Besides our plug-in-result, the main technical tool we use is an isometric formula for the Skorohod integral on the Poisson space. In the context of normal approximation, this approach is new, although it has previously been applied in \cite{PecTh13} for studying approximations by Gamma random variables. In a next step, this is applied to derive a new quantitative version of de Jong's theorem for degenerate U-statistics based on a Poisson measure. As far as we know, this is the first Berry-Essen-type version of de Jong's theorem. In a particular case we shall show that the speed of convergence of the quotient of the fourth moment and the squared variance to $3$ -- which is de Jong's original 
condition to ensure a central limit theorem -- also controls the rate of convergence measured by the Kolmogorov distance. As a second main application, we shall consider Poisson functionals having a finite chaotic expansion. Examples for such functionals are provided by non-degenerate U-statistics. We then study the normal approximation of such (suitably normalized) functionals and provide concrete applications to geometric functionals of intersection processes of Poisson $k$-flats and random graph statistic of the Boolean model as considered in stochastic geometry and to empirical means and second-order moments of Ornstein-Uhlenbeck L\'evy processes. In this context, our new bound simplifies considerably the necessary computations and avoids a subtle technical issue, which is present in \cite{SchulteKolmogorov}. One of the main technical tools is again an isometric formula for the Skorohod integral.

Our text is structured as follows: In Section \ref{sec:preliminaries} we introduce the necessary notions and notation and recall some important background material in order to make the paper self-contained. Our general bound for the normal approximation of Poisson functionals is the content of Section \ref{sec:generalbounds}. Our applications are presented in Section \ref{sec:applications}. In particular, Section \ref{sec:multipleintegrals} deals with multiple Poisson integrals, Section \ref{sec:deJong} with de Jong's theorem for degenerate U-statistics and Section \ref{sec:FiniteExpansion}  with non-degenerate U-statistics and general Poisson functionals having a finite chaotic expansion as well as with our concrete application to stochastic geometry and L\'evy processes.

\section{Preliminaries}\label{sec:preliminaries}

\paragraph{Poisson random measures.}
Let $(\cZ,\sZ)$ be a standard Borel space, which is equipped with a $\sigma$-finite measure $\mu$. By $\eta$ we denote a Poisson (random) measure on $\cZ$ with control $\mu$, which is defined on an underlying probability space $(\Omega,\cF,\PP)$. That is, $\eta=\{\eta(B):B\in\sZ_0\}$ is a collection of random variables indexed by the elements of $\sZ_{0}=\{B\in\sZ:\mu(B)<\infty\}$ such that $\eta(B)$ is Poisson distributed with mean $\mu(B)$ for each $B\in\sZ_0$, and for all $n\in\NN$, the random variables $\eta(B_1),\ldots,\eta(B_n)$ are independent whenever $B_1,\ldots,B_n$ are disjoint sets from $\sZ_0$ (the second property follows automatically from the first one if the measure $\mu$ does not have atoms, cf.\ \cite[Theorem VI.5.16]{Cinlar} or \cite[Corollary 3.2.2]{SW}). The distribution of $\eta$ (on the space of $\sigma$-finite counting measures on $\cZ$) will be denoted by $\PP_\eta$. For more details see \cite[Chapter VI]{Cinlar} and \cite[Chapter 3]{SW}.

\paragraph{$L^1$- and $L^2$-spaces.}
For $n\in\NN$ let us denote by $L^1(\mu^n)$ and $L^2(\mu^n)$ the space of integrable and square-integrable functions with respect to $\mu^n$, respectively. The scalar product and the norm in $L^2(\mu^n)$ are denoted by $\langle\,\cdot\,,\,\cdot\,\rangle_n$ and $\|\cdot\|_n$, respectively. From now on, we will omit the index $n$ as it will always be clear from the context. Moreover, let us denote by $L^2(\PP_\eta)$ the space of square-integrable functionals of a Poisson random measure $\eta$. Finally, we denote by $L^2(\PP,L^2(\mu))$ the space of jointly measurable mappings $h:\Omega\times\cZ\to\RR$ such that $\int_{\Omega} \int_{\cZ}h(\omega, z)^2\,\mu(\dint z) \,\PP(\dint \omega)<\infty$ (recall that $(\O,\cF,\PP)$ is our underlying probability space).

\paragraph{Chaos expansion.}
It is a crucial feature of a Poisson measure $\eta$ that any $F\in L^2(\PP_\eta)$ can be written as
\begin{equation}\label{eq:ChaosExpansion}
F=\EE F+\sum_{n=1}^\infty I_n(f_n)\,,
\end{equation}
where the sum converges in the $L^2$-sense, see \cite[Theorem 1.3]{LastPenrose}. Here, $I_n$ stands for the $n$-fold Wiener-It\^o integral (sometimes called Poisson multiple integral) with respect to the compensated Poisson measure $\eta-\mu$ and for each $n\in\NN$, $f_n$ is a uniquely determined symmetric function in $L^2(\mu^n)$ (depending, of course, on $F$). In particular, the multiple integrals are centered random variables and orthogonal in the sense that $$\EE\big[I_{q_1}(f_1)I_{q_2}(f_2)\big]={\bf 1}(q_1=q_2)\,q_1!\langle f_1,f_2\rangle$$ all for integers $q_1,q_2\geq 1$ and symmetric functions $f_1\in L^2(\mu^{q_1})$ and $f_2\in L^2(\mu^{q_2})$. The representation \eqref{eq:ChaosExpansion} is called the chaotic expansion of $F$ and we say that $F$ has a finite chaotic expansion if only finitely many of the functions $f_n$ are non-vanishing. In particular, \eqref{eq:ChaosExpansion} together with the orthogonality of multiple stochastic integrals leads to the variance formula
\begin{equation}\label{eq:Varianzformel}
\var(F)=\sum_{n=1}^\infty n!\|f_n\|^2\,.
\end{equation}

\paragraph{Malliavin operators.}
For a functional $F=F(\eta)$ of a Poisson measure $\eta$ let us introduce the difference operator $D_zF$ by putting
\begin{equation}\label{eq:DifferenceOperatorPathwise}
 D_zF(\eta) := F(\eta+\d_z) - F(\eta)\,,\qquad z\in\cZ\,.
\end{equation}
$D_zF$ is also called the add-one-cost operator as it measures the effect on $F$ of adding the point $z\in\cZ$ to $\eta$. If $F$ has a chaotic representation as at \eqref{eq:ChaosExpansion} such that $\sum_{n=1}^\infty n\,n!\|f_n\|^2<\infty$ (we write $F\in\dom(D)$ in this case), then $D_zF$ can alternatively be characterized as $$D_zF=\sum\limits_{n=1}^\infty nI_{n-1}(f_n(z,\,\cdot\,))\,,$$ where $f_n(z,\,\cdot\,)$ is the function $f_n$ with one of its arguments fixed to be $z$. We remark that $DF$ is an element of $L^2(\PP,L^2(\mu))$. Besides of $D$, let us also introduce three other Malliavin operators $L$, $L^{-1}$ and $\d$. If $F$ satisfies $\sum_{n=1}^\infty n^2\,n!\|f_n\|^2<\infty$, the Ornstein-Uhlenbeck generator is defined by 
$$LF=-\sum_{n=1}^\infty nI_n(f_n)$$
and its inverse is denoted by $L^{-1}$. In terms of the chaos expansion of a centred random variable $F \in L^2(\PP_{\eta})$,
i.e. $\EE(F)=0$, it is given by $$L^{-1}F=-\sum_{n=1}^\infty{1\over n}I_n(f_n)\,.$$ Finally, if $z\mapsto h(z)$ is a random function on $\cZ$ with chaos expansion $h(z)=z+\sum_{n=1}^\infty I_n(h_n(z,\,\cdot\,))$ with symmetric 
functions $h_n(z,\,\cdot\,)\in L^2(\mu^n)$ such that $\sum_{n=0}^\infty(n+1)!\|h_n\|^2<\infty$ (let us write $h\in\dom(\d)$ if this is satisfied), the Skorohod integral $\d(h)$ of $h$ is defined as $$\d(h)=\sum_{n=0}^\infty I_{n+1}(\widetilde{h}_n)\,,$$ where $\widetilde{h}_n$ is the canonical symmetrization of $h_n$ as a function of $n+1$ variables. The next lemma summarizes a relationships between the operators $D$, $\d$ and $L$, the classical and a modified integration-by-parts-formula (taken from \cite[Lemma 2.3]{SchulteKolmogorov}) as well as an isometric formula for Skorohod integrals, which is Proposition 6.5.4 in \cite{PrivaultBook}.
\begin{lemma}
\begin{itemize}
\item[(i)] For every $F \in \dom(L)$ it holds that $F \in \dom(D)$ and $DF \in \dom(\delta)$, and 
\begin{equation}\label{eq:DeltaD=-L}
 \delta (DF)=-LF\,.
\end{equation}
\item[(ii)] We have the integration-by-parts-formula
\begin{equation}\label{eq:IntegrationByParts}
 \EE [ F \delta(h) ]=\EE\langle DF,h\rangle
\end{equation}
for every $F\in\dom(D)$ and $h\in\dom(\d)$. 
\item[(iii)] Suppose that $F\in L^2(\PP_\eta)$ (not necessarily assuming that $F$ belongs to the domain of $D$), that $h\in\dom(\delta)$ has a finite chaotic expansion and that $D_z{\bf 1}(F>x)h(z)\geq 0$ for any $x\in\RR$ and $\mu$-almost all $z\in\cZ$. Then
\begin{equation}\label{eq:IntegrationByPartsModified}
 \EE [ {\bf 1}(F>x) \delta(h) ]=\EE\langle D{\bf 1}(F>x),h\rangle\,.
\end{equation}
\item[(iv)] If $h\in\dom(\d)$ it holds that
\begin{equation}\label{eq:IsometricFormula}
 \EE[\delta(h)^2] = \EE\int\limits_{\cZ}h(z_1)^2\,\mu(\dint z_1)+\EE\int\limits_{\cZ}\int\limits_{\cZ}(D_{z_2}h(z_1))^2\,\mu(\dint z_1)\mu(\dint z_2)\,.
\end{equation}
\end{itemize}
\end{lemma}
We refer the reader to \cite{NualartVives} or \cite{PSTU2010} for more details and background material concerning the Malliavin formalism on the Poisson space. Moreover, we refer to \cite{LastPenrose} for a pathwise interpretation of the Skorohod integral.

\paragraph{Contractions.}
Let for integers $q_1,q_2\geq 1$, $f_1\in L^2(\mu^{q_1})$ and $f_2\in L^2(\mu^{q_2})$ be symmetric functions and $r\in\{0,\ldots,q\}$, $\ell\in\{1,\ldots,r\}$. The contraction kernel $f_1\s_r^\ell f_2$ on $\cZ^{q_1+q_2-r-\ell}$ acts on the tensor product $f_1\otimes f_2$ first by identifying $r$ variables and then integrating out $\ell$ among them. More formally,
\begin{equation*}
\begin{split}
f_1\s_r^\ell f_2(\g_1,\ldots,\g_{r-\ell},t_1,\ldots,t_{q_1-r},&s_1,\ldots,s_{q_2-r}) = \int\limits_{\cZ^\ell}f_1(z_1,\ldots,z_\ell,\g_1,\ldots,\g_{r-\ell},t_1,\ldots,t_{q_1-r})\\
&\times\,f_2(z_1,\ldots,z_\ell,\g_1,\ldots, \g_{r-\ell},s_1,\ldots,s_{q_2-r})\,\mu^\ell\big(\dint(z_1,\ldots,z_\ell)\big).
\end{split}
\end{equation*}
In addition, we put
\begin{equation*}
f_1\s_r^0f_2(\g_1,\ldots,\g_r,t_1,\ldots,t_{q_1-r},s_1,\ldots,s_{q_2-r})=f_1(\g_1,\ldots,\g_r,t_1,\ldots,t_{q_1-r})f_2(\g_1,\ldots,\g_r,s_1,\ldots,s_{q_2-r}).
\end{equation*}
Besides of the contraction $f_1\s_r^\ell f_2$, we will also deal with their canonical symmetrizations $f_1\ts_r^\ell f_2$. They are defined as $$(f_1\ts_r^\ell f_2)(x_1,\ldots,x_{q_1+q_2-r-\ell})={1\over (q_1+q_2-r-\ell)!}\sum_\pi(f_1\s_r^\ell f_2)(x_{\pi(1)},\ldots,x_{\pi(q_1+q_2-r-\ell)}),$$ where the sum runs over all $(q_1+q_2-r-\ell)!$ permutations of $\{1,\ldots,q_1+q_2-r-\ell\}$.

\paragraph{Product formula.}
Let $q_1,q_2\geq 1$ be integers and $f_1\in L^2(\mu^{q_1})$ and $f_2\in L^2(\mu^{q_2})$ be symmetric functions. In terms of the contractions of $f_1$ and $f_2$ introduced in the previous paragraph, one can express the product of $I_{q_1}(f_1)$ and $I_{q_2}(f_2)$ as follows:
\begin{equation}\label{eq:ProductFormula}
I_{q_1}(f_1)I_{q_2}(f_2)=\sum_{r=0}^{\min(q_1,q_2)}r!{q_1\choose r}{q_2\choose r}\sum_{\ell=0}^r{r\choose\ell}I_{q_1+q_2-r-\ell}(f_1\ts_r^\ell f_2)\,;
\end{equation}
see \cite[Proposition 6.5.1]{PeccatiTaqquBook}.

\paragraph{Technical assumptions.}
Whenever we deal with a multiple stochastic integral, a sequence $F_n=I_q(f_n)$ or a finite sum $F_n=\sum_{i=1}^k I_{q_i}(f_n^{(i)})$ of such integrals with integers $k\geq 1$, $q_i\geq 1$ for $i=1, \ldots,k$, and symmetric functions $f_n\in L^2(\mu_n^q)$ or $f_n^{(i)}\in L^2(\mu_n^{q_i})$ we will (implicitly) assume that the following technical conditions are satisfied (for sequences of single integrals, the upper index has to be ignored):
\begin{itemize}
 \item[i)] for any $i\in \{1, \ldots, k\}$ and any $r\in\{1,\ldots,q_i\}$, the contraction $f_n^{(i)}\s_r^{q_i-r}f_n^{(i)}$ is an element of $L^2(\mu_n^{q_i})$;
 \item[ii)] for any $r\in\{1,\ldots,q_i\}$, $\ell\in\{1,\ldots,r\}$ and $(z_1,\ldots,z_{2q_i-r-\ell})\in\cZ^{2q_i-r-\ell}$ we have that $(|f_n^{(i)}|\s_r^\ell|f_n^{(i)}|)(z_1,\ldots,z_{2q_i-r-\ell})$ is well defined and finite;
 \item[iii)] for any $i,j\in\{1,\ldots,k\}$ and $k\in\{\max(|q_i-q_j|,1),\ldots,q_i+q_j-2\}$ and any $r$ and $\ell$ satisfying $k=q_i+q_j-2-r-\ell$ we have that $$\int\limits_{\cZ}\Big(\int\limits_{\cZ^k}\big(f_n^{(i)}(z,\,\cdot\,)\s_r^\ell f_n^{(j)}(z,\,\cdot\,)\big)^2\,\dint\mu^{k}\Big)^{1/2}\mu(\dint z)<\infty\,,\qquad i,j\in\{1,\ldots,k\}\,.$$
\end{itemize}
For a detailed explanation of the r\^ole of these conditions we refer to \cite{LRP} or \cite{PSTU2010}, but we remark that these technical assumptions ensure in particular that $F_n^2$ is an element of $L^2(\PP_\eta)$, such that $\{\EE\big[F_n^4\big]:n\in\NN\}$ is a bounded sequence. We finally note that (iii) is automatically satisfied if the control measure $\mu$ of the Poisson measure $\eta$ is finite -- just apply the Cauchy-Schwarz inequality.

\paragraph{Probability metrics.} To measure the distance between the distributions of two random variables $X$ and $Y$ defined on a common probability space $(\Omega,\cF,\PP)$, one often uses distances of the form $$d_{\cH}(X,Y)=\sup_{h\in\cH}\big|\EE h(X)-\EE h(Y)\big|\,,$$ where $\cH$ is a suitable class of real-valued test functions (note that we slightly abuse notation by writing $d(X,Y)$ instead of $d({\rm law}(X),{\rm law}(Y))$). Prominent examples are the class $\cH_W$ of Lipschitz functions with Lipschitz constant bounded by one or the class $\cH_K$ of indicator functions of intervals $(-\infty, x]$ with $x\in\RR$. The resulting distances $d_W:=d_{\cH_W}$ and $d_K:=d_{\cH_K}$ are usually called Wasserstein and Kolmogorov distance. We notice that $d_W(X_n,Y)\to 0$ or $d_K(X_n,Y)\to 0$ as $n\to\infty$ for a sequence of random variables $X_n$ implies convergence of $X_n$ to $Y$ in distribution (the converse is not necessarily true, but holds for the Kolmogorov distance if the target random variable $Y$ 
has a density with respect to the Lebesgue measure on $\RR$). 

\paragraph{Stein's method.}
A standard Gaussian random variable $Z$ is characterized by the fact that for every absolutely continuous function $f:\RR\to\RR$ for which $\EE\big[Zf(Z)\big]<\infty$ it holds that $$\EE\big[f'(Z)-Zf(Z)\big]=0\,.$$ This together with the definition of the Kolmogorov distance is the motivation to study the Stein equation
\begin{equation}\label{eq:SteinEquation}
 f'(w)-wf(w) = {\bf 1}(w\leq x) - \Phi(x)\,,\qquad w\in\RR\,,
\end{equation}
in which $x \in \RR$ is fixed and $\Phi(x)=\PP(Z\leq x)$ denotes the distribution function of $Z$. A solution of the Stein equation is a function $f_x$, depending on $x$, which satisfies \eqref{eq:SteinEquation}. The bounded solution of the Stein equation is given by $$f_x(w)=e^{w^2/2}\int\limits_{-\infty}^w\big({\bf 1}(y\leq x)-\Phi(x)\big)e^{-y^2/2}\,\dint y\,,$$ see \cite[Lemma 2.2]{ChenGoldsteinShao}. It has the property that $0 < f_x(w)\leq{\sqrt{2\pi}\over 4}$. Moreover, we observe that $f_x$ is continuous on $\RR$, infinitely differentiable on $\RR\setminus\{x\}$, but not differentiable at $x$. However, interpreting the derivative of $f$ at $x$ as $1 - \Phi(x)+xf(x)$ in view of \eqref{eq:SteinEquation}, we have
\begin{equation}\label{eq:SteinBoundFirstDerivative}
 \big|f_x'(w)\big| \leq 1\qquad{\rm for\ all}\qquad w\in\RR
\end{equation}
according to \cite[Lemma 2.3]{ChenGoldsteinShao}. Moreover, let us recall from the same result that $f_x$ satisfies the bound
\begin{equation}\label{eq:SteinBoundIncrements}
 \big|(w+u)f_x(w+u)-(w+v)f_x(w+v)\big|\leq \left(|w|+{\sqrt{2\pi}\over 4}\right)(|u|+|v|)
\end{equation}
for all $u,v,w\in\RR$.

If we replace $w$ by a random variable $W$ and take expectations in the Stein equation \eqref{eq:SteinEquation}, we infer that $$\EE\big[f_x'(W)-Wf_x(W)\big]=\PP(W\leq x)-\Phi(x)$$ and hence
\begin{equation}\label{eq:SteinKolmogorov}
 \sup_{x\in\RR}\big|\PP(W\leq x)-\Phi(x)\big|=\sup_{x\in\RR}\big|\EE\big[f_x'(W)-Wf_x(W)\big]\big|\,.
\end{equation}
We identify the quantity on the left hand side of \eqref{eq:SteinKolmogorov} as the Kolmogorov distance between (the laws of) $W$ and the standard Gaussian variable $Z$.

\section{General Malliavin-Stein bounds}\label{sec:generalbounds}

Our first contribution in this paper is a new bound for the Kolmogorov distance $d_K(F,Z)$ between a Poisson functional $F$ and a standard Gaussian random variable $Z$. Our bound involves the Malliavin operators $D$ and $L^{-1}$ as introduced in the previous section. Moreover, our set-up is that $(\cZ,\sZ)$ is a standard Borel space which is equipped with a $\sigma$-finite measure $\mu$ and that $\eta$ is a Poisson measure on $\cZ$ with control $\mu$.

\begin{theorem}\label{thm:MalliavinSteinBound}
Let $F\in L^2(\PP_\eta)$ be such that $\EE F=0$ and $F\in\dom(D)$ and denote by $Z$ a standard Gaussian random variable. Then
\begin{equation*}
\begin{split}
d_K(F,Z) \leq \EE|1-\langle DF,-DL^{-1}F\rangle| +{\sqrt{2\pi}\over 8}\,&\EE\langle(DF)^2,|DL^{-1}F|\rangle+{1\over 2}\EE\langle(DF)^2,|F\times DL^{-1}F|\rangle\\
&+\sup_{x\in\RR}\EE\langle (DF)\,D{\bf 1}(F>x),|DL^{-1}F|\rangle\,,
\end{split}
\end{equation*}
where we use the standard notation that $$\langle DF,-DL^{-1}F\rangle=-\int\limits_{\cZ}(D_zF)\times (D_zL^{-1}F)\,\mu(\dint z)$$ (and similarly for the other terms).
\end{theorem}

\begin{remark}\label{rem:AfterThm31}\rm
\begin{itemize}
\item Comparing our bound with that for $d_K(F,Z)$ from \cite{SchulteKolmogorov}, we see that the result in \cite{SchulteKolmogorov} involves the additional term $\EE\langle(DF)^2,|DF\times DL^{-1}F|\rangle$, implying that the bound in \cite{SchulteKolmogorov} is strictly larger than ours. In addition, Theorem \ref{thm:MalliavinSteinBound} improves the constants ibidem. 
\item Our bound should be compared with a similar bound from \cite[Theorem 3.1]{PSTU2010} for the Wasserstein distance between $F$ and $Z$. Namely, if $F\in\dom(D)$ and $\EE F=0$, then Theorem 3.1 in \cite{PSTU2010} states that $$d_W(F,Z)\leq\EE|1-\langle DF,-DL^{-1}F\rangle|+\EE\langle (DF)^2,|DL^{-1}F|\rangle\,.$$ Our bound involves additional term, reflecting the effect that our test functions are indicator functions of intervals $(-\infty,x]$, $x\in\RR$, in contrast to Lipschitz functions with Lipschitz constant bounded by one. 
\item It is well known that Wasserstein and Kolmogorov distance are related by
\begin{equation}\label{eq:InequalityDkDw}
d_K(F,Z)\leq 2\sqrt{d_W(F,Z)}\,. 
\end{equation}
However, this inequality leads to bounds for $d_K(F,Z)$, which are systematically larger than the bounds obtained from Theorem \ref{thm:MalliavinSteinBound}. For instance, if the control of $\eta$ is given by $n\mu$ with integers $n\geq 1$ and if we denote the Poisson functional by $F_n$ in order to indicate its dependence on $n$, then we often have that $d_W(F_n,Z)\leq c_W n^{-1/2}$ and $d_K(F_n,Z)\leq c_Kn^{-1/2}$ for constants $c_W,c_K>0$, whereas \eqref{eq:InequalityDkDw} would deliver the suboptimal rate $n^{-1/4}$ for $d_K(F_n,Z)$ only (see Examples \ref{ex:flats}, \ref{ex:BooleanModel} and \ref{ex:levy} for instance).
\item Other examples for bounds between the law of a Poisson functional and some target random variable in the spirit of Theorem \ref{thm:MalliavinSteinBound} are the paper \cite{PeccatiZhengMulti} dealing with the multivariate normal approximation (with applications in \cite{LPST}), the paper \cite{PecTh13} considering the approximation by a Gamma random variable as well as \cite{PeccatiChenStein}, in which the Chen-Stein method for Poisson approximation has been investigated (see also \cite{STScaling,STFlats} for applications of this result).
\item We finally remark that if $F$ is a functional of a Gaussian random measure on $\cZ$ with control $\mu$, then $$d_K(F,Z)\leq\EE|1-\langle DF,-DL^{-1}F\rangle|$$ as shown in \cite[Theorem 3.1]{NourdinPeccati09}. The presence of additional terms in Theorem \ref{thm:MalliavinSteinBound} are due to the fact that on the Poisson space the Malliavin derivative $D$ is characterized as difference operator, recall \eqref{eq:DifferenceOperatorPathwise}.
\end{itemize}
\end{remark}

\begin{proof}[Proof of Theorem \ref{thm:MalliavinSteinBound}]
Fix some $z\in\cZ$ and $x \in \RR$ and denote by $f:=f_x$ the solution of the Stein equation \eqref{eq:SteinEquation}. Let us first re-write $D_zf(F)$ as
\begin{equation}\label{eq:Dzf(F)}
\begin{split} 
D_zf(F) &= f(F_z) - f(F) = \int\limits_0^{D_zF}f'(F+t)\,\dint t \\
&= \int\limits_0^{D_zF}\big(f'(F+t)-f'(F)\big)\,\dint t+(D_zF)f'(F)
\end{split}
\end{equation}
(note that this is not influenced by the fact that $f'$ only exists as a left- or right-sided derivative at $t=x$). Next, applying \eqref{eq:DeltaD=-L} and the integration-by-parts-formula \eqref{eq:IntegrationByParts} in this order yields
\begin{equation}\label{eq:ProofEEFf(F)}
\begin{split}
\EE\big[Ff(F)\big] = \EE\big[LL^{-1}Ff(F)\big] = \EE\big[\delta(-DL^{-1}F)f(F)\big] = \EE\langle Df(F),-DL^{-1}F\rangle\,.
\end{split}
\end{equation}
We now replace $w$ by $F$ in the Stein equation \eqref{eq:SteinEquation}, take expectations and use \eqref{eq:Dzf(F)} as well as \eqref{eq:ProofEEFf(F)} to see that
\begin{equation}\label{eq:ProofZwischenschritt1}
\begin{split}
&\EE\big[f'(F)-Ff(F)\big]\\
&= \EE\big[f'(F)-\langle Df(F),-DL^{-1}F\rangle\big]\\
&=\EE\big[f'(F)(1-\langle DF,-DL^{-1}F\rangle)\big]-\EE\big[\langle\int\limits_0^{DF}\big(f'(F+t)-f'(F)\big)\,\dint t,-DL^{-1}F\rangle\big]\,.
\end{split}
\end{equation}
Let us consider for fixed $z\in\cZ$ the integral in the second term. Since $f$ is a solution of the Stein equation \eqref{eq:SteinEquation}, we have that $$f'(F+t)=(F+t)f(F+t)+{\bf 1}(F+t\leq x)-\Phi(x)$$ and that $$f'(F)=Ff(F)+{\bf 1}(F\leq x)-\Phi(x)\,,$$ which leads us to
\begin{equation}\label{eq:ProofZwischenschritt2}
\begin{split}
\int\limits_0^{D_zF}\big(f'(F+t)-f'(F)\big)\,\dint t=\int\limits_0^{D_zF}&\big((F+t)f(F+t)- Ff(F)\big)\,\dint t\\
&+\int\limits_0^{D_zF}\big({\bf 1}(F+t\leq x)-{\bf 1}(F\leq x)\big)\,\dint t=:I_1+I_2\,.
\end{split}
\end{equation}
Now, the integrand in $I_1$ can be bounded by means of \eqref{eq:SteinBoundIncrements}, which yields $$|I_1| \leq \int\limits_0^{|D_zF|}\big(|F|+{\sqrt{2\pi}\over 4}\big)|t|\,\dint t = {1\over 2}(D_zF)^2\,\big(|F|+{\sqrt{2\pi}\over 4}\big)\,.$$ To bound $I_2$, we consider the cases $D_zF<0$ and $D_zF\geq 0$ separately and write
\begin{equation*}
\begin{split}
I_{2,<0}&:={\bf 1}(D_zF<0)\,\int\limits_0^{D_zF}\big({\bf 1}(F+t\leq x)-{\bf 1}(F\leq x)\big)\,\dint t\,,\\
I_{2,\geq 0}&:={\bf 1}(D_zF\geq 0)\,\int\limits_0^{D_zF}\big({\bf 1}(F+t\leq x)-{\bf 1}(F\leq x)\big)\,\dint t\,.
\end{split}
\end{equation*}
For the first term, we have \begin{equation*}
\begin{split}
I_{2,<0} &=-\int\limits_{D_zF}^0{\bf 1}(x<F\leq x-t)\,\dint t\leq-\int\limits_{D_zF}^0{\bf 1}(x<F\leq x-D_zF)\,\dint t\\
&= (D_zF)\,{\bf 1}(D_zF+F\leq x<F)\,.
\end{split}
\end{equation*}
Thus, we arrive at the following estimate for $I_{2,<0}$:
\begin{equation*}
\begin{split}
|I_{2,<0}| &\leq \big|(D_zF)\,{\bf 1}(D_zF+F\leq x<F)\,{\bf 1}(D_zF<0)\big|\\
&= \big|(D_zF)\,({\bf 1}(F>x)-{\bf 1}(D_zF+F > x))\,{\bf 1}(D_zF<0)\big|\\
&= \big|(D_zF)\,({\bf 1}(D_zF+F>x)-{\bf 1}(F > x))\,{\bf 1}(D_zF<0)\big|\\
&= \big|(D_zF)\,D_z{\bf 1}(F>x)\,{\bf 1}(D_zF<0)\big|\\
&= (D_zF)\,D_z{\bf 1}(F>x)\,{\bf 1}(D_zF<0)\,,
\end{split}
\end{equation*}
where the equality in the last line follows by considering the cases $D_zF+F,F\leq x$ and $D_zF+F\leq x<F$ separately (note that the remaining cases cannot contribute). For the second case, similar arguments lead to the upper bound $$|I_{2,\geq 0}|\leq (D_zF)D_z{\bf 1}(F>x){\bf 1}(D_zF\geq 0)\,.$$ Thus, for $I_2=I_{2,<0}+I_{2,\geq 0}$ we have that $$|I_2|\leq (D_zF)\,D_z{\bf 1}(F>x)\,.$$ Together with the bound for $I_1$ and the fact that $\big|f'(w)\big|\leq 1$ for all $w\in\RR$ we conclude from \eqref{eq:ProofZwischenschritt1} the bound
\begin{equation*}
\begin{split}
&\big|\EE\big[f'(F)-Ff(F)\big]\big| \leq \EE\big|1-\langle DF,-DL^{-1}F\rangle\big|+\EE\langle|I_1|+|I_2|,|DL^{-1}F|\rangle\\
&\leq \EE\big|1-\langle DF,-DL^{-1}F\rangle\big|+{\sqrt{2\pi}\over 8}\EE\langle(DF)^2,|DL^{-1}F|\rangle+{1\over 2}\EE\langle(DF)^2,|F\times DL^{-1}F|\rangle\\
&\qquad\qquad+\EE\langle(D_zF)\,D_z{\bf 1}(F>x),|DL^{-1}F|\rangle\,.
\end{split}
\end{equation*}
The final result follows in view of \eqref{eq:SteinKolmogorov} by taking the supremum over all $x\in\RR$.
\end{proof}

Let us draw a consequence of Theorem \ref{thm:MalliavinSteinBound}, which will in our applications below serve as kind of plug-in theorem. It provides a more convenient form of the bound for the Kolmogorov distance, which will be applied in the context of Theorem \ref{thm:MultipleIntegrals} and Theorem \ref{thm:FiniteExpansion}. 

\begin{corollary}\label{cor:BoundNachCS}
Let $F$ and $Z$ be as in Theorem \ref{thm:MalliavinSteinBound}. Then
\begin{equation*}
\begin{split}
d_K(F,Z) \leq \EE|1-\langle DF,-&DL^{-1}F\rangle| +{1\over 2}\left(\EE\langle(DF)^2,(DL^{-1}F)^2\rangle\right)^{1/2}\left(\EE\|DF\|^4\right)^{1/4}\big((\EE F^4)^{1/4}+1\big)\\
&+\sup_{x\in\RR}\EE\langle (DF)\,D{\bf 1}(F>x),|DL^{-1}F|\rangle\,.
\end{split}
\end{equation*}
\end{corollary}
\begin{proof}
Since $\sqrt{2\pi}/8<1/2$, the result follows immediately by applying to $${\sqrt{2\pi}\over 8}\,\EE\langle(DF)^2,|DL^{-1}F|\rangle+{1\over 2}\EE\langle(DF)^2,|F\times DL^{-1}F|\rangle\leq{1\over 2}\EE\langle(DF)^2,(1+|F|)|DL^{-1}F|\rangle$$ twice the Cauchy-Schwarz inequality and the Minkowski inequality, then by using the bound provided by Theorem \ref{thm:MalliavinSteinBound}.
\end{proof}


As a particular case, let us consider a multiple integral of arbitrary order:

\begin{example}\label{ex:singularintegral}
\rm
If $F$ has the form $F=I_q(f)$, with $q \geq 1$ and $f \in L^2(\mu^q)$ symmetric, we have by definition that $L^{-1}F = - \frac 1q I_q(f)$
and hence $\langle DF, -DL^{-1}F \rangle = \frac 1q \|DF \|^2$. The second term of the bound in Theorem \ref{thm:MalliavinSteinBound} reads as $ \frac 1q \int_{\cZ} \EE [|D_zF|^3] \, \mu(\dint z)$, multiplied by the constant $\frac{\sqrt{2 \pi}}{8}$, the third term is given by $ \frac 1q \int_{\cZ} \EE [ |F| \, |D_zF|^3] \, \mu(\dint z)$ and the fourth equals  $\sup_{x\in\RR} \frac1q \EE\langle (DF)(D{\bf 1}(F>x)),|DF|\rangle\,$. Moreover, we obtain
$$
{1\over 2}\left(\EE\langle(DF)^2,(DL^{-1}F)^2\rangle\right)^{1/2} = \frac{1}{2q} \left( \int_{\cZ} \EE [ (D_zF)^4] \, \mu(\dint z) \right)^{1/2}.
$$
Applying Jensen's inequality we find $\EE\|DF\|^4 \leq \int_{\cZ} \EE [ (D_zF)^4] \, \mu(\dint z)$. Hence, the
second term of the bound in Corollary \ref{cor:BoundNachCS} can be estimated from above by
$$
\frac{1}{2q} \left( \int_{\cZ} \EE [ (D_zF)^4] \, \mu(\dint z) \right)^{3/4} \, \big((\EE F^4)^{1/4}+1\big)\,.
$$
This set-up will further be exploited in Section \ref{sec:multipleintegrals} below. We refer the reader to \cite[Theorem 3.1]{PSTU2010} for a similar bound for the Wasserstein distance between $I_q(f)$ and $Z$.
\end{example}

\section{Applications}\label{sec:applications}

\subsection{Multiple integrals}\label{sec:multipleintegrals}

In this section we consider a sequence of multiple integrals $F_n:=I_q(f_n)$ for a fixed integer $q\geq 2$ and with functions $f_n\in L^2(\mu^q)$ satisfying the technical assumptions presented in Section 2. Moreover, we shall assume that for each $n\in\NN$, $\eta_n$ is a Poisson random measure on $(\cZ,\sZ)$ with control $\mu_n$, where for each $n\in\NN$, $\mu_n$ is a $\sigma$-finite measure on $\cZ$. In what follows, norms and scalar products involving functions $f_n$ are always taken with respect to $\mu_n$.

\begin{theorem}\label{thm:MultipleIntegrals}
Assume the set-up described above (in particular that $\{f_n:n\in\NN\}$ is a sequence of symmetric functions satisfying the technical assumptions) and suppose that
\begin{equation}\label{eq:AssumptionsContractions}
 \lim\limits_{n\to\infty}q!\|f_n\|^2=1\qquad{\rm and}\qquad \lim_{n\to\infty}\|f_n\s_r^\ell f_n\|=0
 \end{equation}
for all pairs $(r,\ell)$ such that either $r=q$ and $\ell=0$ or $r\in\{1,\ldots,q\}$ and $\ell\in\{0,\ldots,\min(r,q-1)\}$. Then $F_n$ converges in distribution to a standard Gaussian random variable $Z$ and for any $n\in\NN$ we have the following bound on the Kolmogorov distance between $F_n$ and $Z$: $$d_K(F_n,Z)\leq C\times\max\left\{|1-q!\|f_n\|^2\big|,\|f_n\s_r^\ell f_n\|,\|f_n\s_r^\ell f_n\|^{3/2}\right\}$$ with a constant $C>0$ only depending on $q$ and where the maximum runs over all $r$ and $\ell$ such that either $r=q$ and $\ell=0$ or $r\in\{1,\ldots,q\}$ and $\ell\in\{1,\ldots,\min(r,q-1)\}$.
\end{theorem}

\begin{remark}\rm
\begin{itemize}
\item Note that the assumption and the estimate for the Kolmogorov distance in Theorem \ref{thm:MultipleIntegrals} involve the contraction kernel $f_n\s_q^0 f_n=f_n^2$. In particular, the condition that $\|f_n^2\|\to 0$ as $n \to \infty$ is actually a condition on the $L^4$-norm of $f_n$.
\item Under condition \eqref{eq:AssumptionsContractions} we have that $\|f_n\s_r^\ell f_n\|^{3/2}$ is smaller than $\|f_n\s_r^\ell f_n\|$ for sufficiently large indices $n$ so that $d_K(F_n,Z)$ is asymptotically dominated by $\|f_n\s_r^\ell f_n\|$ or the variance difference $|\var(Z)-\var(F_n)|=|1-q!\|f_n\|^2\big|$.
\item It is worth comparing our bound with the one from \cite[Theorem 4.2]{PSTU2010} for the Wasserstein distance: $$d_W(F_n,Z)\leq C_W\times \max\left\{|1-q!\|f_n\|^2\big|,\|f_n\s_r^\ell f_n\|\right\}$$ with a constant $C_W>0$ only depending on $q$ and where the maximum is running over all $(r,\ell)$ satisfying $r=q$ and $\ell=0$ or $r\in\{1,\ldots,q\}$ and $\ell\in\{1,\ldots,\min(r,q-1)\}$. Thus, $d_W(F_n,Z)$ coincides with $d_K(F_,Z)$ up to a constant multiple under condition \eqref{eq:AssumptionsContractions} for sufficiently large $n$. See also \cite[Theorem 3.5]{LRP}.
\item A bound for $d_K(F_n,Z)$ with $F_n=I_q(f_n)$ as in Theorem \ref{thm:MultipleIntegrals} could in principle also be derived using the techniques provided by \cite{SchulteKolmogorov}. However, this leads to an expression which is systematically larger than ours as it involves contractions of the \textit{absolute value} of $f_n$.
\item Similar statements for sequences of multiple integrals with respect to a Gaussian random measure can be found in \cite[Proposition 3.2]{NourdinPeccati09} for instance. In this case, it is sufficient that $q!\|f_n\|^2\to 1$ and that $\|f_n\s_r^r f_n\|\to 0$, as $n\to\infty$, to conclude a central limit theorem for them. Note, that in the Poisson case, assumption \eqref{eq:AssumptionsContractions} involves also contractions $f_n\s_r^\ell f_n$ with $r\neq\ell$, which for general $f_n$ seems unavoidable (see however \cite{PeccatiZhengUniversality} for the case of so-called homogeneous sums, where it suffices to control $\|f_n\s_r^r f_n\|$). 
\end{itemize}
\end{remark}

\begin{proof}[Proof of Theorem \ref{thm:MultipleIntegrals}]
Let us introduce the sequences
\begin{equation*}
\begin{split}
A_1(F_n) := & \EE|1-\langle DF_n,-DL^{-1}F_n\rangle|\,,\\
A_2(F_n) := & \left(\EE\langle(DF_n)^2,(DL^{-1}F_n)^2\rangle\right)^{1/2}\,,\\
A_3(F_n) := & \left(\EE\|DF_n\|^4\right)^{1/4}\times \big((\EE F_n^4)^{1/4}+1\big)\,,\\
A_4(F_n) := & \sup_{x\in\RR}\EE\langle (DF_n)(D{\bf 1}(F_n>x)),|DL^{-1}F_n|\rangle\,,
\end{split}
\end{equation*}
where here and below $F_n$ stands for $I_q(f_n)$ (recall that in our set-up the norms and scalar products are with respect to $\mu_n$). Then Corollary \ref{cor:BoundNachCS} delivers the bound $$d_K(F_n,Z)\leq A_1(F_n)+{1\over 2}A_2(F_n)\times A_3(F_n)+A_4(F_n)\,.$$ 
Thus, we shall show that $A_1(F_n)$, $A_2(F_n)\times A_3(F_n)$, and $A_4(F_n)$ vanish asymptotically, as $n\to\infty$. For $A_1(F_n)$, we use Theorem 4.2 in \cite{PSTU2010}, in particular Equation (4.14) ibidem, to see that
\begin{equation*}
\begin{split}
A_1(F_n) \leq \big|1-q!\|f_n\|^2\big|+q\sum_{r=1}^q\sum_{\ell=1}^{\min(r,q-1)}&{\bf 1}(2\leq r+\ell\leq 2q-1)\,\big((2q-r-\ell)!\big)^{1/2}\,(r-1)!\\
&\times{q-1\choose r-1}^2{r-1\choose \ell-1}\,\|f_n\s_r^{\ell}f_n\|\,.
\end{split}
\end{equation*}
Next, for $A_2(F_n)$ we observe that
\begin{equation}\label{eq:A2Simplification}
\EE\langle(DF_n)^2,(DL^{-1}F_n)^2\rangle=q^{-2} \, \EE \int\limits_{\cZ}(D_zF_n)^4\,\mu_n(\dint z)\,,
\end{equation}
using that, by definition, $DL^{-1}F_n = - \frac 1q D F_n$ (see Example \ref{ex:singularintegral}).
Hence, we can apply again Theorem 4.2 in \cite{PSTU2010}, this time Equation (4.32) and (4.18) ibidem, to deduce the bound
\begin{equation*}
\begin{split}
A_2(F_n) \leq q\,\sum_{r=1}^q\sum_{\ell=0}^{r-1} &{\bf 1}(1\leq r+\ell\leq 2q-1)\,\big((r+\ell-1)!\big)^{1/2}\,(q-\ell-1)! \\
& \times{q-1\choose q-1-\ell}^2{q-1-\ell\choose q-r}\,\|f_n\s_r^\ell f_n\|\,.
\end{split}
\end{equation*}
Concerning $A_3(F_n)$, let us first write $$A_3(F_n)=\left(\EE\|DF_n\|^4\right)^{1/4}\times\big((\EE F_n^4)^{1/4}+1\big)=:A_3^{(1)}(F_n)\times A_3^{(2)}(F_n)\,.$$
Now, use Jensen's inequality to see that
$$A_3^{(1)}(F_n)^4=\EE\Big(\;\int\limits_{\cZ}(D_zF_n)^2\,\mu_n(\dint z)\Big)^2\leq\EE \int\limits_{\cZ} (D_zF_n)^4\,\mu_n(\dint z)  
= (q \, A_2(F_n))^2\,.$$ 
Hence, we conclude that $A_3^{(1)}(F_n) \leq  \bigl( q \, A_2(F_n) \bigr)^{1/2}$. Moreover $A_3^{(2)}(F_n)$ is a bounded sequence, since the functions $f_n$ satisfy the technical assumptions. Finally, let us consider the sequence $A_4(F_n)$. We will adapt in parts the strategy of the proof of Proposition 2.3 in \cite{PecTh13} to derive a bound for $A_4(F_n)$. First, define the mapping $u\mapsto\Xi(u):=u|u|$ from $\RR$ to $\RR$ and observe that it satisfies the estimate
\begin{equation}\label{eq:EstimateXi}
\big(\Xi(v)-\Xi(u)\big)^2\leq 8u^2(v-u)^2+2(v-u)^4
\end{equation}
for all $u,v\in\RR$. To apply the modified integration-by-parts-formula \eqref{eq:IntegrationByPartsModified} we need to check that $D_z{\bf 1}(F_n>x)\,\Xi(D_zF_n)\,|DL^{-1}F_n|\geq 0$. Therefore and in view of the definition of $\Xi$, it is sufficient to show that $(D_z{\bf 1}(F_n>x))(D_zF_n)\geq 0$. To prove this, consider the two cases $F\leq D_zF+F$ and $F>D_zF+F$ separately. In the first case we have $D_zF\geq 0$ and $D_z{\bf 1}(F_n>x)\in\{0,1\}$, whereas in the second case it holds that $D_zF< 0$ along with $D_z{\bf 1}(F_n>x)\in\{-1,0\}$. Thus, $(D_z{\bf 1}(F_n>x))(D_zF_n)\geq 0$, and hence $D_z{\bf 1}(F_n>x)\,\Xi(D_zF_n)\,|DL^{-1}F_n|\geq 0$.
This allows us to apply the modified integration-by-parts-formula \eqref{eq:IntegrationByPartsModified} and to conclude that
\begin{equation}\label{eq:BoundA3IntermediateEstimate}
\begin{split}
A_4(F_n) &= {1\over q}\EE\int\limits_{\cZ}(D_z{\bf 1}(F_n>x))\,\Xi(D_zF_n)\,\mu_n(\dint z)\\
&= {1\over q}\EE\big[{\bf 1}(F_n>x)\,\d(\Xi(DF_n))\big]\\
&\leq {1\over q}\big(\EE\big[\d(\Xi(DF_n))^2\big]\big)^{1/2}\,.
\end{split}
\end{equation}
Now, the Skorohod isometric formula \eqref{eq:IsometricFormula} yields
\begin{equation}\label{eq:BoundA3IntermediateEstimate2}
\begin{split}
\EE\big[\d(\Xi(DF_n))^2\big] \leq \EE\int\limits_{\cZ}\Xi(D_zF_n)^2\,\mu_n(\dint z)+\EE\int\limits_{\cZ}\int\limits_{\cZ}\big(D_{z_2}\Xi(D_{z_1}F_n)\big)^2\,\mu_n(\dint z_1)\mu_n(\dint z_2)\\
= \EE\int\limits_{\cZ}(D_zF_n)^4\,\mu_n(\dint z)+\EE\int\limits_{\cZ}\int\limits_{\cZ}\big(D_{z_2}\Xi(D_{z_1}F_n)\big)^2\,\mu_n(\dint z_1)\mu_n(\dint z_2)\,.
\end{split}
\end{equation}
Since $D_{z_2}\Xi(D_{z_1}F_n)=\Xi(D_{z_1}F_n+D_{z_2}D_{z_1}F_n)-\Xi(D_{z_1}F_n)$, we can apply \eqref{eq:EstimateXi} with $u=D_{z_1}F_n$ and $v=D_{z_2}D_{z_1}F_n + D_{z_1}F_n$ 
there, to see that $$\big(D_{z_2}\Xi(D_{z_1}F_n)\big)^2\leq 8(D_{z_1}F_n)^2(D_{z_2}D_{z_1}F_n)^2+2(D_{z_2}D_{z_1}F_n)^4\,.$$ Combining this with \eqref{eq:BoundA3IntermediateEstimate} and\eqref{eq:BoundA3IntermediateEstimate2} gives
\begin{equation*}
\begin{split}
A_4(F_n) \leq &{2\sqrt{2}\over q}\Bigg(\Big(\EE\int\limits_{\cZ}(D_zF_n)^4\,\mu_n(\dint z)\Big)^{1/2}+\Big(\EE\int\limits_{\cZ}\int\limits_{\cZ}(D_{z_1}F_n)^2(D_{z_2}D_{z_1}F_n)^2\,\mu_n(\dint z_1)\mu_n(\dint z_2)\Big)^{1/2}\\
&\qquad\qquad +\Big(\EE\int\limits_{\cZ}\int\limits_{\cZ}(D_{z_2}D_{z_1}F_n)^4\,\mu_n(\dint z_1)\mu_n(\dint z_2)\Big)^{1/2}\Bigg)\\
&=: {2\sqrt{2}\over q}\left(A_4^{(1)}(F_n)+A_4^{(2)}(F_n)+A_4^{(3)}(F_n)\right)\,.
\end{split}
\end{equation*}
Clearly, $A_4^{(1)}(F_n)= q A_2(F_n)$, 
recall \eqref{eq:A2Simplification}. As seen in the proof of Lemma 4.3 in \cite{PecTh13}, the term $A_4^{(3)}(F_n)$ can be bounded by linear combinations of quantities of the type $\|f_n\s_a^bf_n\|$ with $a\in\{2,\ldots,q\}$ and $b\in\{0,\ldots,a-2\}$. Moreover, the middle term $A_4^{(2)}(F_n)$ can, by means of the Cauchy-Schwarz inequality, be estimated as follows: $$A_4^{(2)}(F_n)\leq \big(q \, A_2(F_n)\big)^{1/2}\times \big(A_4^{(4)}(F_n)\big)^{1/4}$$ 
with $A_4^{(4)}(F_n)$ given by $$A_4^{(4)}(F_n)=\EE\int\limits_{\cZ}\Big(\;\int\limits_{\cZ}(D_{z_2}(D_{z_1}F_n))^2\,\mu_n(\dint z_2)\Big)^2\,\mu_n(\dint z_1)\,.$$ Arguing again as at \cite[Page 549]{PecTh13}, one infers that $A_4^{(4)}(F_n)$ is bounded by linear combinations of $\|f_n\s_a^b f_n\|^2$ with $a$ and $b$ as above.

Consequently, our assumptions \eqref{eq:AssumptionsContractions} imply that $d_K(F_n,Z)\to 0$, as $n\to\infty$. This yields the desired convergence in distribution of $F_n$ to $Z$. The precise bound for $d_K(F_n,Z)$ follows implicitly from the computations performed above. 
\end{proof}

\begin{corollary}\label{cor:4thMomentBound}
Fix an integer $q\geq 2$ and assume that $\{f_n:n\in\NN\}$ is a sequence of non-negative, symmetric functions in $L^2(\mu^q)$ which satisfy the technical assumptions. In addition, suppose that $\EE I_q^2(f_n)=1$ for all $n\in\NN$. Then, there is a constant $C>0$ only depending on $q$ such that for sufficiently large $n$,
\begin{equation}\label{eq:Ungleichung4thMoment}
d_K(I_q(f_n),Z)\leq C\times\sqrt{\EE I_q^4(f_n)-3}\,,
\end{equation}
where $Z$ is a standard Gaussian random variable. Moreover, if the sequence $\{I_q(f_n):n\in\NN\}$ is uniformly integrable, then $I_q(f_n)$ converges in distribution to a standard Gaussian random variable if and only if $\EE I_q^4(f_n)$ converges to $3$.
\end{corollary}
\begin{proof}
The first part follows directly by combining Proposition 3.8 in \cite{LRP} with Theorem \ref{thm:MultipleIntegrals}. 
The second part is Theorem 3.12 (3) in \cite{LRP}.
\end{proof}

\begin{remark}\rm
\begin{itemize}
\item Corollary \ref{cor:4thMomentBound} should be compared with the following result from \cite{NualartPeccati}: Let for some integer $q\geq 2$, $I_q^G(f_n)$ be a sequence of multiple integrals with respect to a Gaussian random measure on $\cZ$ such that for each, $n\in\NN$, $f_n\in L^2(\mu^q)$ is symmetric (but not necessarily non-negative). In addition, suppose that $\EE I_q^G(f_n)^2=1$. Then the convergence in distribution of $I_q^G(f_n)$ to a standard Gaussian random variable is equivalent to the convergence of $\EE[I_q^G(f_n)^4]$ to $3$.
\item The fourth moment criterion stated in Corollary \ref{cor:4thMomentBound} is in the spirit of fourth moment criteria for central limit theorems of Gaussian multiple integrals first obtained in \cite{NualartPeccati} and recalled above. They have attracted considerable interest in recent times and we refer to the webpage $$\text{http://www.iecn.u-nancy.fr/$\sim$nourdin/steinmalliavin.htm}$$ for an exhaustive collection of works in this direction.
\item Inequality \eqref{eq:Ungleichung4thMoment} with Kolmogorov distance $d_K$ replaced by Wasserstein distance $d_W$ has been proved in \cite{LRP}, see Equation (3.9) ibidem.
\item If we have $\EE I_q^2(f_n)\to 1$, as $n\to\infty$, instead of $\EE I_q^2(f_n)=1$, then \eqref{eq:Ungleichung4thMoment} has to be replaced by $d_K(I_q(f_n),Z)\leq C\times\sqrt{\EE I_q^4(f_n)-3(\EE I_q^2(f_n))^2}$. However, this generalization will not be needed in our applications below.
\item The second assertion of Corollary \ref{cor:4thMomentBound} remains true without the assumption that the functions $f_n$ are non-negative in the case of double Poisson integrals (i.e.\ if $q=2$). This is the main result in \cite{PeccatiTaqquDoubleIntegrals}. Because of the involved structure of the fourth moment of a Poisson multiple integral (resulting from a highly technical so-called diagram formula, see \cite{PeccatiTaqquBook}), it is not clear weather a similar result should also be expected for $q>2$.
\end{itemize}
\end{remark}

\subsection{A quantitative version of de Jong's theorem}\label{sec:deJong}

Let ${\bf Y}=\{Y_i:i\in\NN\}$ be a sequence of i.i.d.\ random variables in $\RR^d$ for some $d\geq 1$ whose distribution has a Lebesgue density $p(x)$. Moreover, let, independently of ${\bf Y}$, $\{N_n:n\in\NN\}$ be a sequence of random variables such that each member $N_n$ follows a Poisson distribution with mean $n$. Then, for each $n\in\NN$, $$\eta_n:=\sum_{i=1}^{N_n}\d_{Y_i}$$ is a Poisson random measure on $\cZ =\RR^d$ (equipped
with the standard Borel $\sigma$-field) with (finite) control $\mu_n(\dint x)=np(x)\dint x$ (where $\dint x$ stands for the infinitesimal element of the Lebesgue measure in $\RR^d$). For convenience we put $\mu:=\mu_1$. Next, let for each $n\in\NN$, $h_n:\RR^{2d}\to\RR$ be a non-zero, symmetric function, which is integrable with respect to $\mu^2$. By a sequence of (bivariate) U-statistics (sometimes called Poissonized $U$-statistics) based on these data we understand a sequence $\{U_n:n\in\NN\}$ of Poisson functionals of the form $$U_n:=\sum_{(Y,Y')\in\eta_{n,\neq}^2}h_n(Y,Y')\,,$$ where $\eta_{n,\neq}^2$ is the set of all distinct pairs of points of $\eta_n$. $h_n$ is called kernel function of $U_n$. We shall assume that these U-statistics are completely degenerate in the sense that for any $n \in \NN$,
$$\int\limits_{\RR^d} h_n(x,y)\, \mu(\dint x) = \int\limits_{\RR^d} h_n(x,y)\,p(x)\,\dint x=0\qquad\mu{\rm -a.e.}$$ It is well known that a completely degenerated $U_n$ can be represented as $U_n=I_2(f_{2,n})$ with $f_{2,n}=h_n$. 
It is a direct consequence of \eqref{eq:Varianzformel} that $$\var(U_n) = 2 n^2 \,\EE[h_n(Y_1,Y_2)^2]\,,$$
where the expectation $\EE$ is the integral with respect to $\mu^2$.
Let us also introduce the normalized U-statistic $F_n:=U_n/\sqrt{\var(U_n)}$. Our main result in this section is a quantitative version of de Jong's theorem \cite{DeJonge} for such U-statistics.

\begin{theorem}\label{thm:DeJonge}
Let $\{h_n:n\geq 1\}$ be as above and suppose that $h_n\in L^4(\mu^2)$ as well as 
\begin{equation}\label{eq:assumptionDeJong}
\sup_{n\in\NN}{\int_{\RR^d}h_n^4\,\dint\mu_n^2\over \left(\int_{\RR^d}h_n^2\,\dint\mu_n^2\right)^2}<\infty\,.
\end{equation}
Then the fourth moment condition
\begin{equation}\label{eq:4thMomentCondition}
\lim_{n \to \infty} \EE F_n^4 = \lim_{n\to\infty}{\EE[U_n^4]\over (\var(U_n))^2}=3
\end{equation}
implies that $F_n$ converges in distribution to a standard Gaussian random variables $Z$. Moreover, there exists a universal constant $C>0$ such that for all $n$,
\begin{equation*}
\begin{split}
d_K(F_n,Z)\leq & C\times {1\over\var(U_n)}\times\max\left\{\|h_n\s_2^0h_n\|,\|h_n\s_1^1h_n\|,\|h_n\s_2^1h_n\|,\right.\\
&\qquad\qquad\qquad\qquad\qquad\qquad\left.\|h_n\s_2^0h_n\|^{3/2},\|h_n\s_1^1h_n\|^{3/2},\|h_n\s_2^1h_n\|^{3/2}\right\}\,.
\end{split}
\end{equation*}
\end{theorem}

\begin{remark}\rm
\begin{itemize}
\item The set-up of this section fits into our general framework by taking $\cZ=\RR^d$ and $\sZ$ as its Borel $\sigma$-field.
\item The first assertion of Theorem \ref{thm:DeJonge} corresponds de Jong's theorem in \cite{DeJonge}. Whereas the original proof is long and technical, our proof is more transparent and directly deals with the fourth moment. It is the slightly corrected version of the proof taken from \cite{PecTh13}. On the other hand, the technique in \cite{DeJonge} also allows to deal with U-statistics whose kernel functions $h_n$ are not necessarily symmetric. 
\item Theorem \ref{thm:DeJonge} is a generalization of (the corrected form of) Theorem 2.13 (A) in \cite{PecTh13}, which deals with the Wasserstein distance between $F_n$ and $Z$. In fact, the bound for $d_W(F_n,Z)$ coincides -- up to a constant multiple --  with the bound for $d_K(F_n,Z)$. To the best of our knowledge, Theorem \ref{thm:DeJonge} is the first quantitative version of de Jong's theorem, which deals with the Kolmogorov distance.
\item The paper \cite{PecTh13} also contains a quantitative version of de Jong's theorem, where the target random variable follows a Gamma distribution instead of a standard Gaussian distribution. In this case, the probability metric is based on the class of trice differentiable test functions. 
\end{itemize}
\end{remark}

\begin{proof}[Proof of Theorem \ref{thm:DeJonge}]
Since $U_n$ is completely degenerate, we can represent the normalized U-statistic $F_n$ as $I_2(f_n)$ with $f_n=h_n/\sqrt{\var(U_n)}$ (note that the double Poisson integral is taken with respect to the compensated Poisson measure $\eta_n-\mu_n$). The estimate for $d_K(F_n,Z)$ is thus a consequence of Theorem \ref{thm:MultipleIntegrals}. We shall show that in fact $d_K(F_n,Z)$ tends to zero as $n\to\infty$ if \eqref{eq:4thMomentCondition} is satisfied. Using the product formula \eqref{eq:ProductFormula} and the orthogonality of multiple integrals together with the relation $$4!\|f_n\ts_0^0f_n\|^2=2(2\|f_n\|^2)^2+16\|f_n\s_1^1f_n\|^2$$ from \cite[Equation 5.2.12]{NourdinPeccatiBook}, we see that
\begin{equation}\label{eq:ZwischenschrittDeJong}
\begin{split}
\EE F_n^4=16\times 3! &\|f_n\ts_1^0f_n\|^2+16\|f_n\s_2^1f_n\|^2+16\|f_n\s_1^1f_n\|^2\\
&+2\|4f_n\s_1^1f_n+2f_n^2\|^2+3(2\|f_n\|^2)^2\,,
\end{split}
\end{equation}
where, as usual in this section, norms and contractions are with respect to $\mu_n$ (observe that to verify these computations, assumption \eqref{eq:assumptionDeJong} is essential). For some more details on how to obtain this relation we refer the reader to \cite[Formulae (4.12) and (4.13)]{PecTh13}. Since $\var(F_n)=2\|f_n\|^2=1$ by construction, we clearly have $3(2\|f_n\|^2)^2=3$ for all $n\in\NN$. Thus, if the fourth moment condition \eqref{eq:4thMomentCondition} is satisfied, the other (non-negative) terms in \eqref{eq:ZwischenschrittDeJong} must vanish asymptotically, as $n\to\infty$. Consequently, $d_K(F_n,Z)$ tends to zero, as $n\to\infty$.
\end{proof}

Let us finally in this section present a version of de Jong's theorem, were the speed of convergence in the fourth moment condition \eqref{eq:4thMomentCondition} also controls the rate of convergence of $F_n$ towards a standard Gaussian random variable.

\begin{corollary}
Assume the same set-up as in Theorem \ref{thm:DeJonge} and suppose in addition that $h_n$ is non-negative for each $n\in\NN$. Then there is a universal constant $C>0$ such that for sufficiently large $n$, $$d_K(F_n,Z)\leq C\times\sqrt{\EE F_n^4-3}=C\times\sqrt{{\EE[U_n^4]\over \var(U_n)^2}-3}\,,$$ where $Z$ is a standard Gaussian random variable.
\end{corollary}
\begin{proof}
This is consequence of Theorem \ref{thm:DeJonge} and Corollary \ref{cor:4thMomentBound}. Note that the assumption $\EE F_n^2=1$ for each $n\in\NN$ is automatically fulfilled by construction.
\end{proof}

\subsection{Functionals with finite chaotic expansion}\label{sec:FiniteExpansion}

Let us assume that $(\cZ,\sZ)$ is a standard Borel space, $\{\mu_n:n\in\NN\}$ is a sequence of $\sigma$-finite measures on $\cZ$ and for each $n\in\NN$, $\eta_n$ is a Poisson random measure with control $\mu_n$. In this section we deal with a sequence $\{F_n:n\in\NN\}$ of Poisson functionals such that for each $n\in\NN$, $F_n$ admits the representation
\begin{equation}\label{eq:FiniteExpansion}
F_n=\sum_{i=1}^k I_{q_i}(f_n^{(i)})
\end{equation}
with integers $1=q_1<\ldots<q_k$ ($k\in\NN$) and symmetric functions $f_n^{(i)}\in L^2(\mu_n^{q_i})$. Note that each of the multiple integrals $I_{q_i}$, $i\in\{1,\ldots,k\}$, is taken with respect to $\eta_n-\mu_n$. We shall assume that for all $n\in\NN$ and $i\in\{1,\ldots,k\}$, the functions $f_n^{(i)}$ satisfy the technical assumptions and are such that $\|f_n^{(1)}\|>0$. In particular, this implies that $F_n\in L^2(\PP_{\eta_n})$.

A particular interesting class of such functionals are non-degenerate U-statistics of Poisson random measures. To define them, let, as above, $k\geq 1$ be  a fixed integer and $h\in L^1(\mu^k)$ be a symmetric function. Then a U-statistic based on $\eta_n$ and $h$ is given by 
$$U_n:=\sum_{(z_1,\ldots,z_k)\in\eta_{n,\neq}^k}h(z_1,\ldots,z_k),$$
where the symbol $\eta_{n,\neq}^k$ indicates the class of all $k$-dimensional vectors $(z_1,\ldots,z_k)$ such that $z_i \in \eta_n$ and $z_i \not= z_j$ for every $1 \leq i \not= j \leq k$. 
We always assume that $U_n\in L^2(\PP_{\eta_n})$ (then necessarily $h \in L^2(\mu^k)$), in which case $U_n$ can be re-written as $$U_n=\EE U_n+\sum_{i=1}^k I_i(g_n^{(i)})$$ with $$g_n^{(i)}(z_1,\ldots,z_i)={k\choose i}\int\limits_{\cZ^{k-i}}h(z_1,\ldots,z_i,y_1,\ldots,y_{k-i})\,\mu_n^{k-i}\big(\dint(y_1,\ldots,y_{k-i})\big)\,$$ for every $i =1, \ldots, k$, see  \cite[Lemma 3.5]{ReitznerSchulte}. Moreover, the multivariate Mecke formula \cite[Corollary 3.2.3]{SW} implies that $$\EE U_n=\int\limits_{\cZ^k}h(z_1,\ldots,z_k)\,\mu_n^k\big(\dint(z_1,\ldots,z_k)\big)\,.$$ One should note that this chaotic representation follows from an application of the results proved in \cite{LastPenrose}. In contrast to the situation considered in the previous section around de Jong's theorem and in order to ensure the non-degeneracy of the U-statistics, we will assume that $\|g_n^{(1)}\|>0$ for all $n\in\NN$. Let us finally introduce -- by slight abuse of notation -- the normalized U-statistics $F_n$ by $F_n:=(U_n-\EE U_n)/\sqrt{\var(U_n)}$.  It is easy to see that the so-defined $F_n$ has the chaotic expansion $$F_n=I_1(f_n^{(1)})+\ldots+I_k(f_n^{(k)})$$ with $f_n^{(i)}=g_n^{(i)}/\sqrt{\var(U_n)}$ for $i\in\{1,\ldots,k\}$. 


\begin{theorem}\label{thm:FiniteExpansion}
Consider a Poisson functional as at \eqref{eq:FiniteExpansion} in the set-up as described above and suppose that for all $n\in\NN$ and $i\in\{1,\ldots,k\}$, the functions $f_n^{(i)}$ satisfy the technical assumptions  for all $n\in\NN$ and that $\lim\limits_{n\to\infty}\var(F_n)=1$. Moreover, let $Z$ be a standard Gaussian random variable. Then there is a universal constant $C>0$ such that for all $n$,
\begin{equation*}
\begin{split}
d_K(F_n,Z)\leq C\times &\left(\max\big\{|1-\var{F_n}|,\max\big\{\|f_n^{(i)}\s_r^\ell f_n^{(i)}\|,\|f_n^{(i)}\s_r^\ell f_n^{(i)}\|^{3/2}\big\}\right.\\ &\qquad\qquad\qquad\left.+\max\big\{\|f_n^{(i)}\s_r^\ell f_n^{(j)}\|,\|f_n^{(i)}\s_r^\ell f_n^{(j)}\|^{3/2}\big\}\big\}\right)\,,
\end{split}
\end{equation*}
where the first maximum is taken over all $i\in\{1,\ldots,k\}$ and pairs $(r,\ell)$ such that either $r=q_i$ and $\ell=0$ or $r\in\{1,\ldots,q_i\}$ and $\ell\in\{1,\ldots,\min(r,q_i-1)\}$, whereas the second maximum is taken over all $i,j\in\{1,\ldots,k\}$ with $i<j$ and pairs $(r,\ell)$ satisfying $r\in\{1,\ldots,q_i\}$ and $\ell\in\{1,\ldots,r\}$.
\end{theorem}

\begin{remark}\rm 
\begin{itemize}
\item The statement of Theorem \ref{thm:FiniteExpansion} remains true if we replace Kolmogorov distance by Wasserstein distance, see \cite[Theorem 3.5]{LRP}.
\item A bound for $d_K(F_n,Z)$ has also been derived in \cite{SchulteKolmogorov} in the case of U-statistics, see Theorem 4.2 there. However, this bound is systematically larger than our bound, but not only because of an additional term in Theorem \ref{thm:MalliavinSteinBound} (recall Remark \ref{rem:AfterThm31}). It also involves (after re-writing the terms $M_{ij}$ there in our language) contractions of the \textit{absolute values} of the functions $f_n^{(i)}$ and $f_n^{(j)}$. This goes hand in hand with the observation that the notion of absolute convergence of U-statistics introduced and used in \cite{ReitznerSchulte,SchulteKolmogorov} can be avoided in our framework.
\end{itemize}
\end{remark}

\begin{corollary}\label{cor:UStatistics4thMoment}
We assume the same framework as in Theorem \ref{thm:FiniteExpansion} and suppose in addition that $f_n^{(i)}\geq 0$ for all $n\in\NN$ and $i\in\{1,\ldots,k\}$. Then there is a constant $C>0$ such that for sufficiently large $n$, $$d_K(F_n,Z)\leq C\times\sqrt{\EE F_n^4-3(\EE F_n^2)^2}\,,$$ where $Z$ is a standard Gaussian random variable. Moreover, if the sequence $\{F_n:n\in\NN\}$ is uniformly integrable, convergence in distribution of $F_n$ to $Z$ is equivalent to convergence of $\EE F_n^4-3(\EE F_n^2)^2$ to $0$.
\end{corollary}
\begin{proof}
This is a consequence of Theorem \ref{thm:FiniteExpansion} and Proposition 3.8 and Theorem 3.12 in \cite{LRP}.
\end{proof}

\begin{remark}\rm
\begin{itemize}
\item Corollary \ref{cor:UStatistics4thMoment} is a direct generalization of Corollary \ref{cor:4thMomentBound}, which deals with sequence of single multiple integrals. 
\item With $d_K$ replaced by $d_W$, the fourth moment bound has already been stated in \cite{LRP}. Moreover, in the special case $q_1=1,\ldots,q_k=k$, corresponding to a U-statistic, the bound for the Kolmogorov distance also appears in \cite{SchulteKolmogorov}.
\item For Corollary \ref{cor:UStatistics4thMoment} to be true, the assumption that the functions $f_n^{(i)}$ are non-negative is essential. It is an open problem whether this can be relaxed.
\end{itemize}
\end{remark}

\begin{proof}[Proof of Theorem \ref{thm:FiniteExpansion}]
We use Corollary \ref{cor:BoundNachCS} to see that $$d_K(F_n,Z)\leq A_1(F_n)+{1\over 2}A_2(F_n)\times A_3(F_n)+A_4(F_n)$$ with
\begin{equation*}
\begin{split}
A_1(F_n) &:= \EE|1-\langle DF_n,-DL^{-1}F_n\rangle|\,,\\
A_2(F_n) &:= \left(\EE\langle(DF_n)^2,(DL^{-1}F_n)^2\rangle\right)^{1/2}\,,\\
A_3(F_n) &:= \left(\EE\|DF_n\|^4\right)^{1/4}\big((\EE F_n^4)^{1/4}+1\big)\,,\\
A_4(F_n) &:= \sup_{x\in\RR}\EE\langle (DF_n)(D{\bf 1}(F_n>x)),|DL^{-1}F_n|\rangle\,,
\end{split}
\end{equation*}
where norms and scalar products are always taken with respect to $\mu_n$. To bound $A_1(F_n)$ we use the first part of Theorem 3.5 in \cite{LRP} (which is a consequence of Proposition 5.5 in \cite{PeccatiZhengMulti}), which yields $$A_1(F_n)\leq C_1\times\left(\max\big\{|1-\var{F_n}|,\max\left\{\|f_n^{(i)}\s_r^\ell f_n^{(i)}\|\right\}+\max\left\{\|f_n^{(i)}\s_r^\ell f_n^{(j)}\|\right\}\big\}\right)$$ for some constant $C_1>0$. Here, the first maximum is taken over all $i\in\{1,\ldots,k\}$ and pairs $(r,\ell)$ such that $r\in\{1,\ldots, q_i\}$ and $\ell\in\{1,\ldots,\min(r, q_i-1)\}$ (note that the case $r=q_i$ and $\ell=0$ is excluded here), whereas the second maximum is taken over all $i,j\in\{1,\ldots,k\}$ with $i<j$ and pairs $(r,\ell)$ satisfying $r\in\{1,\ldots,q_i\}$ and $\ell\in\{0,\ldots,r\}$. To bound $A_2(F_n)$, we write
\begin{equation*}
\begin{split}
A_2(F_n)^2 &= \EE\int\limits_{\cZ}(D_zF_n)^2(D_zL^{-1}F_n)^2\,\mu_n(\dint z)\\
&=\EE\int\limits_{\cZ}\left(\sum_{i=1}^k q_iI_{q_i-1}(f_n^{(i)}(z,\,\cdot\,))\right)^2\left(\sum_{j=1}^kI_{q_j-1}(f_n^{(j)}(z,\,\cdot\,))\right)^2\,\mu_n(\dint z)\\
&\leq q_k^2\,\EE\int\limits_{\cZ}\left(\sum_{i=1}^k I_{q_i-1}^2(f_n^{(i)}(z,\,\cdot\,))\right)\left(\sum_{j=1}^kI_{q_j-1}^2(f_n^{(j)}(z,\,\cdot\,))\right)\,\mu_n(\dint z)\,.
\end{split}
\end{equation*}
Now, the product formula \eqref{eq:ProductFormula} allows us to re-write $I_{q_i-1}^2(f_n^{(i)}(z,\,\cdot\,))$ and $I_{q_j-1}^2(f_n^{(j)}(z,\,\cdot\,))$ as a sum of multiple integrals. The orthogonality of these integrals then implies that $A_2(F_n)^2$ is bounded by a linear combination of terms of form $\|f_n^{(i)}\s_r^\ell f_n^{(i)}\|^2$ and $\|f_n^{(i)}\s_r^\ell f_n^{(j)}\|^2$ with $i,j,r$ and $\ell$ as in the statement of the theorem. 
The sequence $A_3(F_n)$ can be bounded as follows. Using our technical assumptions, the factor $(\EE F_n^4)^{1/4}+1$ is bounded.
Moreover, we obtain by Jensen's inequality that
\begin{equation*}
\begin{split}
\EE \| D F_n \|^4 & \leq \EE\int\limits_{\cZ}(D_zF_n)^4\,\mu_n(\dint z) 
=\EE\int\limits_{\cZ}\left(\sum_{i=1}^k q_iI_{q_i-1}(f_n^{(i)}(z,\,\cdot\,))\right)^4 \,\mu_n(\dint z)\\
&\leq q_k^4\,\EE\int\limits_{\cZ}\left(\sum_{i=1}^k I_{q_i-1}^2(f_n^{(i)}(z,\,\cdot\,))\right)
\left(\sum_{j=1}^kI_{q_j-1}^2(f_n^{(j)}(z,\,\cdot\,))\right)
\,\mu_n(\dint z)\,.
\end{split}
\end{equation*}
Therefore $A_3(F_n)$ is bounded by a linear combination of terms of form $\|f_n^{(i)}\s_r^\ell f_n^{(i)}\|^{1/2}$ and $\|f_n^{(i)}\s_r^\ell f_n^{(j)}\|^{1/2}$ with $i,j,r$ and $\ell$ as in the statement of the theorem. It remains to consider $A_4(F_n)$ and we shall follow the strategy of the proof of Theorem \ref{thm:MultipleIntegrals} to derive an estimate for it. To do this, recall that $F_n$ was of the form $F_n=\sum_{i=1}^k I_{q_i}(f_n^{(i)})$, implying that $$D_zF_n=\sum_{i=1}^k q_i\,I_{q_i-1}(f_n^{(i)}(z,\,\cdot\,))\qquad{\rm and}\qquad -D_zL^{-1}F_n=\sum_{i=1}^k I_{q_i-1}(f_n^{(i)}(z,\,\cdot\,))\,.$$ Using the modified integration-by-parts-formula \eqref{eq:IntegrationByPartsModified} together with the fact that ${\bf 1}(F_n>x)\leq 1$ similarly as in the proof of Theorem \ref{thm:MultipleIntegrals}, we see that $$A_4(F_n)\leq \big(\EE\big[\delta((DF_n)|DL^{-1}F_n|)^2\big]\big)^{1/2}\,.$$ Thus, the isometric formula for Skorohod integrals \eqref{eq:IsometricFormula} implies that
\begin{equation*}
\begin{split}
A_4(F_n)&\leq \left(\,\int\limits_{\cZ}\EE\big[(D_zF_n)^2(D_zL^{-1}F_n)^2\big]\,\mu_n(\dint z)\right)^{1/2}\\
&\qquad\qquad+\left(\,\EE\int\limits_{\cZ}\int\limits_{\cZ}\big(D_{z_2}(D_{z_1}F_n|D_{z_1}L^{-1}F_n|)\big)^2\,\mu_n(\dint z_2)\mu_n(\dint z_1)\right)^{1/2}\,.
\end{split}
\end{equation*}
The first term is just $A_2(F_n)$, whereas the second term in brackets is bounded by a linear combination of quantities of the type $$\EE\int\limits_{\cZ}\int\limits_{\cZ}\big(D_{z_2}(D_{z_1}I_{q_i}(f_n^{(i)})|D_{z_1} I_{q_j}(f_n^{(j)})|)\big)^2\,\mu_n(\dint z_2)\mu_n(\dint z_1)\,,$$ where $i$ and $j$ range from $1$ to $k$. To analyse them, let us define the function $\Psi(x,y):=x|y|$ for $x,y\in\RR$ and observe that we can find finite constants $c_1>0$ and $c_2>0$ such that for all $a,b,c,d\in\RR$, $$\big(\Psi(a+c,b+d)-\Psi(a,b)\big)^2\leq c_1\times a^2d^2+c_2\times c^2(b^2+d^2)$$ as a consequence of a multivariate Taylor expansion (as in the proof of Theorem \ref{thm:MultipleIntegrals}, the constants $c_1$ and $c_2$ can be determined explicitly, but are not important for our purposes here). This gives that
\begin{equation*}
\begin{split}
&\big(D_{z_2}(D_{z_1}I_{q_i}(f_n^{(i)})|D_{z_1}I_{q_j}(f_n^{(j)})|)\big)^2\\
&=\big( \Psi(D_{z_1}I_{q_i}(f_n^{(i)})+D_{z_2}D_{z_1}I_{q_i}(f_n^{(i)}),D_{z_1}I_{q_j}(f_n^{(j)})+D_{z_2}D_{z_1}I_{q_j}(f_n^{(j)}))\\
&\qquad\qquad\qquad\qquad-\Psi(D_{z_1}I_{q_i}(f_n^{(i)}),D_{z_1}I_{q_j}(f_n^{(j)}))\big)^2\\
&\leq c_1\times (D_{z_1}I_{q_i}(f_n^{(i)}))^2(D_{z_2}D_{z_1}I_{q_j}(f_n^{(j)}))^2+c_2\times (D_{z_2}D_{z_1}I_{q_i}(f_n^{(i)}))^2(D_{z_1}I_{q_j}(f_n^{(j)}))^2\\
&\qquad\qquad\qquad\qquad+c_2\times (D_{z_2}D_{z_1}I_{q_i}(f_n^{(i)}))^2(D_{z_2}D_{z_1}I_{q_j}(f_n^{(j)}))^2\,.
\end{split}
\end{equation*}
Using the Cauchy-Schwarz inequality we thus conclude that $$A_4(F_n)\leq A_2(F_n)+C_1\times \max_{i,j\in\{1,\ldots,k\}}\left\{\sqrt{A_{i,j}^{(1)}(F_n)}+\sqrt{A_{i,j}^{(2)}(F_n)}+\sqrt{A_{i,j}^{(3)}(F_n)}\right\}$$ with a constant $C_1>0$ and $A_{i,j}^{(1)}(F_n)$, $A_{i,j}^{(2)}(F_n)$ and $A_{i,j}^{(3)}(F_n)$ given by
\begin{equation*}
\begin{split}
&A_{i,j}^{(1)}(F_n):=\EE \int\limits_{\cZ}\int\limits_{\cZ}(D_{z_1}I_{q_i}(f_n^{(i)}))^2(D_{z_2}D_{z_1}I_{q_j}(f_n^{(j)}))^2\,\mu_n(\dint z_2)\mu_n(\dint z_1)\,,\\
&A_{i,j}^{(2)}(F_n):=\EE \int\limits_{\cZ}\int\limits_{\cZ}(D_{z_2}D_{z_1}I_{q_i}(f_n^{(i)}))^2(D_{z_1}I_{q_j}(f_n^{(j)}))^2\,\mu_n(\dint z_2)\mu_n(\dint z_1)\,,\\
&A_{i,j}^{(3)}(F_n):=\EE \int\limits_{\cZ}\int\limits_{\cZ}(D_{z_2}D_{z_1}I_{q_i}(f_n^{(i)}))^2(D_{z_2}D_{z_1}I_{q_j}(f_n^{(j)}))^2\,\mu_n(\dint z_2)\mu_n(\dint z_1)
\end{split}
\end{equation*}
with $i,j\in\{1,\ldots,k\}$. In the proof of Theorem \ref{thm:MultipleIntegrals} we have shown that the first two sequences, $A_{i,j}^{(1)}(F_n)$ and $A_{i,j}^{(2)}(F_n)$, are bounded by linear combinations of squared norms of contractions of $f_n^{(i)}$ and $f_n^{(j)}$. Turning to the last term $A_{i,j}^{(3)}(F_n)$, we apply once more the Cauchy-Schwarz inequality to deduce that 
\begin{equation*}
\begin{split}
A_{i,j}^{(3)}(F_n)\leq &\left(\,\EE\int\limits_{\cZ}\int\limits_{\cZ}(D_{z_2}D_{z_1}I_{q_i}(f_n^{(i)}))^4\,\mu_n(\dint z_2)\mu_n(\dint z_1)\right)^{1/2}\\
&\qquad\qquad\qquad\times \left(\,\EE\int\limits_{\cZ}\int\limits_{\cZ}(D_{z_2}D_{z_1}I_{q_j}(f_n^{(j)}))^4\,\mu_n(\dint z_2)\mu_n(\dint z_1)\right)^{1/2}\,.
\end{split}
\end{equation*}
The terms in brackets can now be bounded as in the proof of Theorem \ref{thm:MultipleIntegrals}.
\end{proof}

We now present three concrete application of Theorem \ref{thm:FiniteExpansion}. The first one deals with a certain class of functionals arising in stochastic geometry \cite{SW} and generalizes the results developed in \cite{ReitznerSchulte}. In particular, we consider a much wilder class of geometric functionals, which is inspired by the findings in \cite{LPST}. In our second example we consider random graph statistics of the Boolean model, which have previously been considered in \cite[Section 8.1]{LRP2}. Our third application deals with non-linear functionals of a certain class of L\'evy processes and generalizes results in \cite{PSTU2010,PeccatiTaqquDoubleIntegrals,PeccatiZhengMulti}. We emphasize that this example does not fit within the class of U-statistics and is hence not in the domain of attraction of the applications considered in \cite{SchulteKolmogorov}. Other examples to which our theory could directly be applied to are counting statistics for random geometric graphs \cite{LRP,PenroseBook} and random simplicial complexes \cite{DecreusefondEtAl} or proximity functionals of non-intersecting flats \cite{STScaling,STFlats}.

\begin{example}\label{ex:flats}\rm
Let $A(d,k)$ be the space of $k$-dimensional affine subspaces of $\RR^d$ ($d\geq 1$ and $k\in\{0,\ldots,d-1\}$) and define the sequence $\{\mu_n:n\in\NN\}$ of $\sigma$-finite measures on $A(d,k)$ by $$\mu_n(\,\cdot\,)=n\int\limits_{G(d,k)}\int\limits_{L^\perp}{\bf 1}(L+x\in\,\cdot\,)\,\cH^{d-k}(\dint x)\QQ(\dint L)\,.$$ Here, $\QQ$ is a probability measure on the space $G(d,k)$ of $k$-dimensional linear spaces of $\RR^d$ and $\cH^{d-k}$ stands for the $(d-k)$-dimensional Hausdorff measure. By $\eta_n$ we denote a Poisson random measure on $A(d,k)$ with intensity measure $\mu_n$.

Let $m\in\NN$ be such that $d-m(d-k)\geq 0$. Then the intersection process $\eta_n^{[m]}$ of order $m$ of $\eta_n$ arises as the collection of all subspaces $E_1\cap\ldots\cap E_m$, where $(E_1,\ldots,E_m)\in\eta_{n,\neq}^m$ are in general position. In terms of \cite{SW}, $\eta_n^{[m]}$ is a translation-invariant process of $(d-m(d-k))$-dimensional subspaces of $\RR^d$. Both, $\eta_n$ as well as $\eta_n^{[m]}$, are one of the classical objects studied in stochastic geometry and we refer to \cite{SW,STFlats} for more details. 

We call a geometric functional every non-negative measurable function $\varphi$ on the space of convex subsets of $\RR^d$ with the properties that $\varphi(\emptyset)=0$ and $|\varphi(K\cap E_1\cap\ldots\cap E_m)|\leq c(K)$ for $\mu_n^m$-almost all $(E_1,\ldots,E_m)$ and all compact convex subsets $K\subset \RR^d$, where $c(K)$ is a constant only depending on $K$. Examples of geometric functionals are
\begin{itemize}
\item the $(d-m(d-k))$-dimensional Hausdorff measure,
\item the counting functional $\varphi(K)={\bf 1}(K\neq\emptyset)$,
\item the intrinsic volume $V_i$ of order $i\in\{0,\ldots,d\}$ (cf.\ \cite[Chapter 14]{SW}),
\item generalized chord-power integrals $V_i(\,\cdot\,)^\alpha$ with $i\in\{0,\ldots,d\}$ and $\alpha\geq 0$, where $V_i$ is the intrinsic volume of order $i$ (note that this functional is not additive),
\item integrals with respect to support measures (or generalized curvature measures) as considered in convex geometry (cf.\ \cite[Chapter 14]{SW} and the references cited therein).
\end{itemize}
Given a geometric functional and the Poisson random measures $\eta_n$ together with its intersection process $\eta_n^{[m]}$ as above, we consider for a compact convex subset $K\subset\RR^d$ the (non-degenerate) U-statistic $$U_n(K)={1\over m!}\sum_{(E_1,\ldots,E_m)\in\eta_{n,\neq}^m}\varphi(K\cap E_1\cap\ldots\cap E_m)$$ (note that the pre-factor $1/m!$ compensates multiple counting in the subsequent sum). Using \eqref{eq:Varianzformel} it is easy to see that, as $n\to\infty$, $\var(U_n)$ behaves asymptotically like $V(\varphi,m,K)\times n^{2m-1}$ with $$V(\varphi,m,K):={1\over(m-1)!}\int\limits_{A(d,k)}\Big(\;\int\limits_{A(d,k)^{m-1}}\varphi(K\cap E_1\cap\ldots\cap E_m)\,\mu_n^{(m-1)}\big(\dint(E_2,\ldots,E_m)\big)\Big)^2\mu_n(\dint E_1)\,,$$ with can possibly further be evaluated for concrete choices of $\varphi$. Moreover, writing $g_n^{(i)}$, $i\in\{1,\ldots,m\}$, for the kernels of the chaotic expansion of $U_n(K)$, it follows directly from the definition of the contraction operator that $\|g_n^{(i)}\s_
r^\ell g_n^{(j)}\|^2$ is proportional to $n^{(m-i)+(m-j)+\ell}$. So, the maximal exponent is realized if $i=j=1$ and $\ell=0$. Consequently, denoting by $F_n(K):=(U_n(K)-\EE U_n(K))/\sqrt{\var(U_n(K))}$ the normalized U-statistic, we have from Theorem \ref{thm:FiniteExpansion} the Berry-Esseen bound $$d_K(F_n(K),Z)\leq C\times \left({n^{2m-2}\over n^{2m-1}}\right)^{1/2} = C\times n^{-1/2}$$ with a constant $C>0$ only depending on $\varphi$, $m$ and $K$.
\end{example}

\begin{example}\label{ex:BooleanModel}\rm
Let $\cK^d$ be the space of compact convex subsets of $\RR^d$ (for some $d\geq 2$) and for each $K\in\cK^d$ we denote by $m(K)$ the center of the smallest circumscribed ball (called midpoint of $K$ in the sequel) and define $\cK_0^d:=\{K\in\cK^d:m(K)=0\}$ as the subspace of compact convex subsets of $\RR^d$ with midpoint at the origin. This allows us to identify $\cK^d$ with the product space $\cK_0^d\times\RR^d$ by identifying each $K\in\cK^d$ with the pair $(K-m(K),m(K))$. Now, let $\mu_0$ be a probability measure on $\cK_0^d$ and for each $n\in\NN$, $\eta_n$ be a Poisson random measure on $\cK^d$ with control $\mu_n$ given by $$\mu(\,\cdot\,)=\lambda\int\limits_{\cK_0^d}\int\limits_{\RR^d}{\bf 1}(K+x\in\,\cdot\,)\,\dint x\,\mu_0(\dint K)\,,$$ where $\dint x$ stands for the infinitesimal element of the Lebesgue measure on $\RR^d$ and $\lambda>0$ is a fixed intensity parameter. The union set $\bigcup_{K\in\eta_n}K$ is the so-called Boolean model associated with $\eta_n$, cf.\ \cite{SW}. It is random closed set in the sense of \cite{SW} if $\int_{\cK_0^d}{\rm vol}(K+C)\,\mu_0(\dint K)<\infty$ for all compact sets $C\subset\RR^d$, where $+$ stands for the usual Minkowski addition and ${\rm vol}(\,\cdot\,)$ for the Lebesgue measure on $\RR^d$. Let us further fix a function $h:\RR^d\to[0,\infty)$ and define the sequence $\{U_n:n\in\NN\}$ of Poisson functionals by $$U_n:=\sum_{(K,K')\in\eta}h(m(K)-m(K'))\,{\bf 1}\big(m(K)\in[-n^{1/d},n^{1/d}]^d,\,m(K')\in[-n^{1/d},n^{1/d}]^d,\,K\cap K'\neq\emptyset\big)\,.$$ To ensure that the $U_n$ have finite second-order moments, we assume that $h^2$ is integrable over any compact subset of $\RR^d$ and, moreover, that $\int_{[-n^{1/d},n^{1/d}]^d}h(x)\,\dint x\neq 0$ for all $n\in\NN$. Standard examples are $h\equiv 1$ or $h(x-y)={\rm dist}(x,y)^\alpha$, where ${\rm dist}(x,y)$ stands for the Euclidean distance of $x$ and $y$ and $\alpha>0$. We see that $U_n$ is a non-negative and non-degenerate U-statistic of order two in the sense of this section and the variance formula \eqref{eq:Varianzformel} says that $\var(U_n)$ behaves like a constant times $n$, as $n\to\infty$. In addition, the multivariate Mecke formula \cite[Corollary 3.2.3]{SW} implies that also $\EE U_n$ behaves, as $n\to\infty$, like a constant times $n$. Moreover, the computations in \cite[Section 8.1]{LRP2} imply that for the normalized U-statistics $F_n=(U_n-\EE U_n)/\sqrt{\var(U_n)}$ we have the Berry-Esseen inequality $$d_K(F_n,Z)\leq C\times n^{-1/2}\,,$$ where $Z$ is a standard Gaussian random variable and $C>0$ is a universal constant (depending on $\mu_0$ and $h$, but not on $n$), provided that $\int_{\RR^d}h(x)^p\,\chi(x)\,\dint x<\infty$ for $p\in\{2,4\}$, where $\chi(x)=\PP(K\cap(K'+x)\neq\emptyset)$, where $K$ and $K'$ are two independent compact convex sets with distribution $\mu_0$.
\end{example}

\begin{remark}\rm
We remark that Example \ref{ex:BooleanModel} does also fit within the framework of so-called stabilizing functionals, which have successfully been considered in geometric probability and stochastic geometry in recent years. Indeed, if the diameter of a random set with distribution $\mu_0$ decays exponentially fast and if the function $h(x)$ converges to zero, as ${\rm dist}(x,0)\to\infty$ at a sub-exponential rate, then $U_n$  is exponentially stabilizing in the sense of \cite{PenroseYukich} with respect to $\eta_n$. It is now interesting to see that the central limit theorem in \cite{PenroseYukich} -- which is also based on Stein's method for normal approximation and serves as a standard reference in this field -- delivers a rate of order $n^{-1/2}\log n$ for the Kolomgorov distance, whereas our technique allows to remove the superfluous logarithmic factor (with the Kolmogorov distance replaced by Wasserstein distance this has previously been observed in \cite{LRP2}).
\end{remark}

\begin{example}\label{ex:levy}\rm
Let $\eta$ be a Poisson random measure on $\RR\times\RR$ with control $\nu(\dint u)\,\dint x$, where $\nu$ is a positive, non-atomic, $\sigma$-finite measure on $\RR$ such that $\int_{\RR}u^k\,\nu(\dint u)<\infty$ for all $k\in\{2,\ldots,6\}$ and $\int_{\RR} u^2\,\nu(\dint u)=1$. The Ornstein-Uhlenbeck-L\'evy process for a parameter $\lambda>0$ based on $\eta$ is defined by $$Y_t^\lambda=\sqrt{2\lambda}\int\limits_{-\infty}^t\int\limits_{\RR} u\,e^{-\lambda(t-x)}\,\widehat{\eta}(\dint u,\dint x)\,,\qquad t\geq 0\,,$$ where $\widehat{\eta}$ is the compensated Poisson measure corresponding to $\eta$, see \cite{BarndorffNielsen,PSTU2010,PeccatiZhengMulti}. Our assumptions on $\nu$ ensure that the process $(Y_t^\lambda)_{t\geq 1}$ is well defined and that $\var(Y_t^\lambda)=1$. We consider the following functionals of $(Y_t^{\lambda})_{t\geq 0}$:
\begin{itemize}
\item the empirical mean $M_T:={1\over \sqrt{T}}\int_0^TY_t^{\lambda}\,\dint t$,
\item the empirical second-order moment $S_T:=\sqrt{T}\left({1\over T}\int_0^T(Y_t^{\lambda})^2\,\dint t-1\right)$,
\item the empirical shifted moment $V_T^{(h)}:=\sqrt{T}\left({1\over T}\int_0^T Y_t^{\lambda} Y_{t+h}^{\lambda}\,\dint t-e^{-\lambda h}\right)$,
\end{itemize}
where $T>0$. Clearly, $S_T=V_T^{(0)}$ so that the empirical shifted moments are a generalization of the empirical second-order moment. It follows from \cite[Example 3.6]{PSTU2010}, \cite[Section 7.1]{PSTU2010} and \cite[Corollary 6.10]{PeccatiZhengMulti} that $M_T$, $S_T$ and $V_T^{(h)}$ can be represented as
\begin{equation*}
M_T=I_1(f_1(T))\,,\qquad
S_T =I_1(f_2(T))+I_2(f_3(T))\,,\qquad
V_T^{(h)} = I_1(f_4(T))+I_2(f_5(T))\,,
\end{equation*}
with suitable symmetric functions $f_1(T),f_2(T),f_3(T),f_4(T),f_5(T)$. Using the computations ibidem and defining $c_\nu:=\int_\RR u^4\,\nu(\dint u)$, we are able to deduce the following Berry-Esseen estimates from Theorem \ref{thm:FiniteExpansion}:
\begin{equation*}
\begin{split}
& d_K\bigg({M_T\over \sqrt{2/\lambda}},Z\bigg) \leq C_M\times T^{-1/2}\,,\\
& d_K\bigg({S_T\over \sqrt{2/\lambda+c_\nu^2}},Z\bigg) \leq C_S\times T^{-1/2}\,,\\
& d_K\bigg({V_T^{(h)}\over \sqrt{2/\lambda +c_\nu^2 e^{-2\lambda h}}},Z\bigg) \leq C_V\times T^{-1/2}
\end{split}
\end{equation*}
with constants $C_M>0$, $C_S>0$ and $C_V>0$ only depending on the parameter $\lambda$, or on $\lambda$ and $h$, respectively, and where $Z$ is a standard Gaussian random variable. 
\end{example}

\subsection*{Acknowledgements}
We would like to thank Kai Krokowski and Anselm Reichenbachs for pointing out an error in the first version of the manuscript as well as Matthias Schulte for a helpful discussion.\\
The authors have been supported by the German research foundation (DFG) via SFB-TR 12.


\end{document}